\documentclass[11pt, a4paper]{amsart}
\usepackage{amsmath}

\newtheorem{theorem}{Theorem}[section]     

\newtheorem{lemma}[theorem]{Lemma}
\newtheorem{proposition}[theorem]{Proposition}
\newtheorem{corollary}[theorem]{Corollary}

\usepackage[colorlinks=true,citecolor=cyan]{hyperref}

\usepackage[vcentermath]{youngtab}
\usepackage{ytableau}
\usepackage[all]{xy}

\newcommand{\Z}{\mathbb Z}
\newcommand{\Styl}{Styl}
\newcommand{\K}{\mathbb K}
\newcommand{\C}{\mathbb C}

\newcommand{\Plax}{Plax}

\usepackage{tikz}
\usetikzlibrary{patterns}
\usetikzlibrary{fadings}
\tikzfading[name=fade out, inner color=transparent!0, outer color=transparent!90]
\usetikzlibrary{backgrounds}

\usepackage[labelfont=normalfont,labelformat=simple]{subcaption}
\usepackage{graphicx}
\usepackage{enumerate}

\usepackage[colorinlistoftodos]{todonotes}

\usepackage{mathrsfs}
\def\J{\mathscr{J}}

\title{Quivers of stylic algebras}

\author{Antoine Abram}
\author{Christophe Reutenauer}
\author{Franco Saliola}
\address{Universit\'e du Qu\'ebec \`a Montr\'eal}
\thanks{The authors acknowledge the support of the Natural Sciences and
    Engineering Research Council of Canada (NSERC). This research was
    facilitated by computer exploration using the open-source mathematical
    software system \textsc{SageMath}~\cite{SageMath} and its algebraic
    combinatorics features developed by the \textsc{Sage-Combinat} community
\cite{sage-combinat}.}

\date{\today}

\begin{document}

\maketitle

\begin{abstract}
    We construct a complete system of primitive orthogonal idempotents and
    give an explicit quiver presentation of the monoid algebra of the stylic
    monoid introduced by Abram and Reutenauer.
\end{abstract}

\maketitle

\tableofcontents

\section{Introduction}

We study the monoid algebra of the stylic monoid $\Styl(A)$ introduced by the
first two authors in \cite{AR}. We begin by recalling its definition.

Let $A$ be a totally ordered finite alphabet
and $A^*$ the free monoid that it generates.
The \emph{Robinson--Schensted--Knuth (RSK) correspondence} associates
with each word $w \in A^*$ a semistandard tableau $P(w)$ with entries in $A$
called its \emph{$P$-symbol}.
If $w$ is a decreasing word, then its $P$-symbol $P(w)$ is a column, which
allows us to identify the set of decreasing words on $A$ with the set
$\Gamma(A)$ of column-shaped tableaux with entries in $A$.
This induces a left action of $A^*$ on $\Gamma(A)$:
for a word $x \in A^*$ and a column $\gamma \in \Gamma(A)$,
take $x \cdot \gamma$ to be the first column of the tableau $P(x w)$, where
$w$ is the decreasing word corresponding to the column $\gamma$.
(This action can be defined using the Schensted column insertion procedure;
see \S\ref{Schensted-column-insertion}.)
The finite monoid of endofunctions of $\Gamma(A)$ obtained by this action is
the \emph{stylic monoid $\Styl(A)$}.

It turns out that $\Styl(A)$ is canonically isomorphic to a quotient of the
celebrated plactic monoid. Recall that the plactic monoid has appeared in many
contexts in algebraic combinatorics and was used to give the first rigorous
proof of the \emph{Littlewood--Richardson rule} \cite{Schutzenberger1977,
Lascoux-Schutzenberger-1981, Lothaire2002}.
The monoid algebra $\K\Styl(A)$, where $\K$ is any field, is the
first example of a finite dimensional representation of the plactic monoid that
does not pass through the abelianisation (to our knowledge).
This article is a first step towards understanding the structure of this
representation.

Stylic monoids are examples of $\J$-trivial monoids \cite{AR}, which are
a ubiquitous class of monoids that arise naturally in algebraic combinatorics.
Other examples include the $0$-Hecke monoids associated with finite Coxeter
groups, and the monoids of regressive order-preserving functions on a poset;
see \cite{DHST} for many more examples.
It follows that the monoid algebra $\K\Styl(A)$ admits
a \emph{quiver presentation}: that is, $\K\Styl(A)$ is isomorphic to a quotient
of the path algebra $\K Q(A)$ of a canonical quiver $Q(A)$.

Obtaining a quiver presentation is an essential step towards applying the tools
and techniques from the modern representation theory of finite dimensional
algebras~\cite{ASSVol1}.
One of our main
results is an explicit presentation of $\K\Styl(A)$ as a \emph{quiver
with relations}.
Our approach is constructive in the sense that we explicitly identify
a complete system of primitive orthogonal idempotents in $\K\Styl(A)$
(Theorem~\ref{system}) that we use to define a quiver $Q(A)$ together with
a surjective map $\varphi: \K Q(A) \xrightarrow{} \K\Styl(A)$
(Corollary~\ref{quiver-map-is-surjective}) whose kernel is an admissible
ideal (Theorem~\ref{span}). General theory then implies that $Q(A)$ is
the quiver of $\K\Styl(A)$ (Theorem~\ref{quiver}).

We remark that the representation theory of finite monoids naturally occurring
in algebraic combinatorics, especially in connection with Markov chains,
has been investigated by many authors:
\cite{BD1998,
BHR1999,
BBD1999,
Brown2000,
Brown2004,
Saliola2007,
Saliola2009,
BBBS2011,
GanyushkinMazorchuk2011,
MargolisSteinberg2011,
MargolisSteinberg2012,
MazorchukSteinberg2012,
HST2013,
GrensingMazorchuk2014,
ASST2015,
MargolisSteinbergSaliola2015,
MargolisSteinberg2018,
Stein2020,
MargolisSteinbergSaliola2021};
see especially Steinberg's recent book and the references therein \cite{SteinbergMonoidRepTheory}.
Those most closely related to the present work are \cite{DHST},
\cite[Chapter~17]{SteinbergMonoidRepTheory} and \cite{MargolisSteinberg2018},
which describe the quiver of the algebra of a $\J$-trivial monoid.
While guided by this work, our approach is complementary and completely
self-contained as their techniques do not involve constructing primitive
orthogonal idempotents or a quiver presentation. In fact, in
\cite{MargolisSteinberg2018} one reads \emph{``It is notoriously difficult to write
down explicit primitive idempotents for monoids algebras (c.f. \cite{BBBS2011,
Denton2011}) and often they have complicated expressions in terms of the monoid
basis, making it virtually impossible to determine even the dimension of the
corresponding projective indecomposable module let alone construct a matrix
representation out of it.''}

\section{Stylic monoid and algebra}

We consider a totally ordered finite set $A$, whose elements are called {\it letters}, and the free monoid $A^*$ that it generates. Its elements are called {\it words}. The {\it alphabet} of a word~$x$ is the set of letters $Alph(x)$ appearing in $x$.

\subsection{Tableaux}
We call a {\it tableau} what is usually called a {\it semistandard Young tableau}: a finite lower order ideal
of the poset $\mathbb N^2$, ordered naturally (that is, a finite subset $E\subset \mathbb N^2$ such that $x\leq y$ and $y\in E$ implies $x\in E$),
together with a weakly increasing mapping into $A$, such that the restriction of this mapping to
each subset with given $x$-coordinate is injective. A tableau is usually represented as in Figure \ref{tab}. The
conditions may be expressed by saying that the letters in $A$ are weakly increasing from left to right in each row, and
strictly
increasing from the bottom to top in each column.

\begin{figure}
\begin{ytableau}
 d\\
b&  b\\
a&a&c \end{ytableau}
\caption{A tableau}\label{tab}
\end{figure}

A {\it column} is a tableau with only one column. The set of columns on $A$ is denoted by $\Gamma(A)$. A column is identified naturally with a subset of $A$, and also with the word in $A^*$ that is the decreasing product of its elements.

\subsection{Schensted's column insertion procedure}
\label{Schensted-column-insertion}
Let us recall the Schensted {\it column insertion} algorithm. Let $\gamma$ be a column, viewed here as a subset
of $A$, and let $x\in A$. There are two cases: if $\forall y\in\gamma, x>y$, then define $\gamma'=
\gamma\cup x$. Otherwise, let $y$ be the smallest element in $\gamma$ with $y\geq x$; then define $
\gamma'=(\gamma\setminus y)\cup x$. Then $\gamma'$ is the column obtained by {\it column insertion of $x$ into} $
\gamma$, and in the second case, $y$ is said to be {\it bumped}.

We define a left action of $A^*$ on $\Gamma(A)$, denoted $u\cdot \gamma$, for each $u\in A^*$ and each
column $\gamma$. Since $A^*$ is the free monoid on $A$, it is enough to define the action for each letter $a\in A$.
Define
$$a\cdot \gamma=\gamma'$$
if $\gamma'$ is obtained from $\gamma$ by  column insertion of $a$ into $\gamma$.

For further use, we note that if $\gamma$ is a column, then we have
$$
\gamma\cdot \emptyset=\gamma,
$$
where on the left-hand side, $\gamma$ is viewed as a decreasing word.

\subsection{Stylic monoid}
\label{Stylic-monoid}
We denote by $\Styl(A)$ the monoid of endofunctions of the set $\Gamma(A)$ of columns obtained by the action defined above. Thus, a typical element of $\Styl(A)$ is a function
\begin{eqnarray*}
    \mu_w : \Gamma(A) & \to & \Gamma(A) \\
    \gamma & \mapsto  & w\cdot \gamma
\end{eqnarray*}
for some word $w\in A^*$. Since $\Gamma(A)$ is finite, $\Styl(A)$ is finite. Let $\mu:A^*\to \Styl(A)$ be the canonical monoid homomorphism defined by $\mu(w) = \mu_w$.

We denote by $\equiv_{styl}$ the
corresponding monoid congruence of $A^*$, called the {\it stylic congruence}:
\begin{equation*}
    u\equiv_{styl} v \quad \Longleftrightarrow \quad \mu(u)=\mu(v)
    \quad \Longleftrightarrow \quad u\cdot \gamma=v\cdot \gamma \text{~for all columns $\gamma$.}
\end{equation*}
The monoid $\Styl(A)$ acts naturally on the set of columns, and we
take the same notation: $m\cdot \gamma=w\cdot \gamma$ if $m=\mu(w)$.

\subsection{Relationship with the plactic monoid}
\label{Relationship-with-the-plactic-monoid}

The {\it Schensted $P$-symbol}  is a mapping that associates with each word $w$ on $A$ a tableau $P(w)$, see \cite{Sagan, Lothaire2002}.
The relation~$\equiv_{plax}$ on $A^*$, defined by
$$
u\equiv_{plax} v \quad \Longleftrightarrow \quad P(u)=P(v),
$$
is a congruence of the monoid $A^*$, called the {\it plactic congruence}. The quotient monoid~$A^*/{\equiv_{plax}}$ is called the {\it plactic monoid}.

The {\it column-reading word} of a tableau is the word obtained by reading the columns from left to right, each column being read as a decreasing word. For example, the column reading word of the tableau from Figure~\ref{tab} is the word $dbabac$. If $T$ is a tableau, with column-reading word $w$, then
\begin{equation}\label{Pw=T}
P(w)=T
\end{equation}
by a theorem of Schensted.

The {\it plactic relations}, due to Knuth, are the following relations:
\begin{equation}\label{plactic-abc}
bac\equiv_{plax} bca, \qquad acb\equiv_{plax}cab
\end{equation}
for any choice of letters $a<b<c$ in $A$, and
\begin{equation}\label{plactic-ab}
bab\equiv_{plax} bba, \qquad aba\equiv_{plax}baa
\end{equation}
for any choice of letters $a<b$ in $A$. The plactic congruence is generated by these relations.

By \cite[Theorem~8.1]{AR}, the stylic congruence is generated by the plactic relations~\eqref{plactic-abc} and~\eqref{plactic-ab} together with the {\it idempotent relations} $a^2=a$ for any letter $a$ in $A$.
It then follows that if $B\subset A$, then there is a natural embedding $\Styl(B)\to \Styl(A)$ \cite[Corollary~8.4]{AR}.

\subsection{$N$-tableaux}
\label{N-tableaux}
According to \cite[Theorem~7.1]{AR}, there is a mapping from $A^*$ into the set of tableaux
that induces a bijection from the stylic monoid $\Styl(A)$ onto the set of {\it $N$-tableaux} on $A$.
The image of $x \in A^*$ is denoted $N(x)$ and is called the {\it $N$-tableau} of $x$. We also denote by $N$ the induced bijection from the stylic monoid onto the set of $N$-tableaux.
The precise definition of $N$ is not needed here, rather we will make use of the following properties of $N$.

\begin{proposition}\label{properties}
    \begin{itemize}[leftmargin=10pt]
        \item
            The first column of $N(x)$ is equal to that of the $P$-symbol $P(x)$, and it is $x\cdot\emptyset$, where $\emptyset$ denotes the empty column.

        \item
            The set of $a\in A$ fixing $w \in \Styl(A)$
            on the left ($aw=w$) is equal
            to the first column of $N(w)$.

        \item
            If $x$ is in $\Styl(A)$, then the column-reading word $w$ of $N(x)$ satisfies
            \begin{equation}\label{xwN}
                x=\mu(w).
            \end{equation}
    \end{itemize}
\end{proposition}

\begin{proof}
    The first statement is \cite[Lemma 7.2 (i)]{AR}.
    The second statement follows from \cite[Theorem~11.4]{AR} and definition of the left $N$-insertion.
    The third statement follows from the analogous statement for row-reading words, which is \cite[Equation~5]{AR}, and the fact that column-reading and row-reading words of the same tableau are plactic-, hence stylic-, equivalent.
\end{proof}

\subsection{The anti-automorphism $\theta$}
\label{theta}

Recall, from \cite[Section~9]{AR}, the involutive anti-automorphism $\theta$ of the monoid $A^*$: when restricted to $A$, it reverses the order of $A$. It extends to an endofunction of $\Gamma(A)$, if one identifies as we do columns on $A$ and subsets of $A$. Since $\theta$ preserves the plactic relations, and the idempotent relations, it induces an anti-automorphism of the monoids $A^*$, $\Plax(A)$ and $\Styl(A)$.

\subsection{Stylic algebra}
\label{Stylic-algebra}

We denote by $\Z\Styl(A)$ the $\Z$-algebra of the stylic monoid, and we call it the {\it stylic algebra over} $\Z$. We shall consider also the stylic algebra over a field $\K$, which we denote by $\K\Styl(A)$.

\begin{lemma}\label{axa} Let $x\in\Z\Styl(A)$ and let $a$ be a letter such that each letter
appearing in $x$ is larger or equal to $a$.

(i) $axa=xa$;

(ii) $(1-a)xa=0$

(iii) $(1-a)x(1-a)=(1-a)x$.
\end{lemma}

\begin{proof} (i) follows from Lemma 9.4 in \cite{AR}.
Next, (ii) and (iii) follow by an evident computation.
\end{proof}

The next lemma extends the plactic relations in \eqref{plactic-ab}.

\begin{lemma}\label{superplax} Let $p,q\geq 1$. Consider letters in $A$ satisfying $x_1<\cdots<x_p<y<z_1<\cdots<z_q$, then
$$(x_1\cdots x_p)(z_1\cdots z_q)y\equiv_{styl}(z_1\cdots z_q)(x_1\cdots x_p)y,$$ and
$$
y(x_1\cdots x_p)(z_1\cdots z_q)\equiv_{styl}y(z_1\cdots z_q)(x_1\cdots x_p).
$$
\end{lemma}

\begin{proof} We prove the first identity by double induction. Suppose first that $q=1$. If $p=1$, we are reduced to the plactic relation $x_1z_1y\equiv_{styl}z_1x_1y.$
Suppose that $p\geq 2$. Then, by the plactic relations, we have
$$(x_1\cdots x_{p-1})x_pz_1y  \equiv_{styl} (x_1\cdots x_{p-1})z_1x_py \equiv_{styl} z_1(x_1\cdots x_{p-1})x_py$$
by induction on $p$ applied to the product $(x_1\cdots x_{p-1})z_1x_p$.

Suppose now that $q\geq 2$. Then, using the congruences $z_qy\equiv_{styl}z_qyy\equiv_{styl}yz_qy$ twice, we have
\begin{align*}
    x_1\cdots x_pz_1\cdots z_qy
    & \equiv_{styl} x_1\cdots x_pz_1\cdots z_{q-1}yz_qy \\
    & \equiv_{styl} z_1\cdots z_{q-1}x_1\cdots x_pyz_qy & \text{(by induction on $q$)} \\
    & \equiv_{styl} z_1\cdots z_{q-1}x_1\cdots x_pz_qy \\
    & \equiv_{styl} z_1\cdots z_{q-1}z_qx_1\cdots x_py & \text{(case $q=1$)}.
\end{align*}
By applying the anti-automorphism $\theta$ to the first identity we obtain
$$
\theta(y)\theta(z_q)\cdots \theta(z_1)\theta(x_p)\cdots\theta(x_1)\equiv_{plax}\theta(y)
\theta(x_p)\cdots\theta(x_1)\theta(z_q)\cdots\theta(z_1).
$$
Note that
$$
\theta(z_q)<\cdots<\theta(z_1)<\theta(y)<\theta(x_p)<\cdots<\theta(x_1).
$$
Hence we obtain the second identity of the lemma by a change of variables, after exchanging $p$ and $q$.
\end{proof}

\section{Primitive idempotents of the stylic algebra}

In this section, we construct a complete system of primitive orthogonal idempotents in the stylic algebra $\Z\Styl(A)$.

Recall that $\Gamma (A)$ denotes the set of columns on the totally ordered finite alphabet~$A$.
Let $\gamma \in \Gamma(A)$ be a column and define
\begin{equation}\label{egamma}
 e_\gamma=\prod^{\nearrow}_{a\notin \gamma}(1-a)\prod^{\searrow}_{a\in\gamma} a\in \Z \Styl(A),
\end{equation}
where the arrows indicate that the first product is in increasing order of letters, and the second in decreasing order.

For future use, we note that the second product in \eqref{egamma} is the image of $\gamma$ (viewed as a word)  in $\Styl(A)$;
since decreasing words are idempotent in $\Styl(A)$ \cite[Theorem 12.1]{AR}, we have
\begin{equation}\label{egammagamma}
e_\gamma \gamma =e_\gamma.
\end{equation}

\begin{theorem}\label{system}
    The idempotents $e_\gamma$, one for each $\gamma\in \Gamma(A)$, form
    a complete system of primitive orthogonal idempotents of $\Z\Styl(A)$.
    Precisely, we have
    \begin{enumerate}
        \item
            $e_\gamma^2 = e_\gamma$ and $e_\gamma e_\delta = 0$ for all $\gamma, \delta \in \Gamma(A)$ with $\delta \neq \gamma$;

        \item
            $\sum_{\gamma \in \Gamma(A)} e_\gamma = 1$;

        \item
            for every $\gamma \in \Gamma(A)$, the idempotent $e_\gamma$
            cannot be written as $e_\gamma = x + y$ with $x$ and $y$ nonzero orthogonal idempotents
            in $\Z\Styl(A)$.
    \end{enumerate}
\end{theorem}

\begin{proof}
1. We show that the elements $e_\gamma$ are orthogonal idempotents, by induction on the cardinality of the alphabet $A$. We use the fact that $\Styl(B)$ embeds canonically in $\Styl(A)$ if $B\subset A$, and similarly for their monoid algebras.

Let $a$ be the smallest letter in $A$. Let $\gamma$ and $\delta$ be two columns on $A$. For $\gamma' \in \Gamma(A \setminus a)$, we denote by $e'_{\gamma'}$ the elements \eqref{egamma} relative to the alphabet $A\setminus a$. We distinguish four cases:

-- If $a\in \gamma\cap \delta$, then by \eqref{egamma},
    $e_\gamma = e'_{\gamma \setminus a} a$ and
    $e_\delta = e'_{\delta \setminus a} a$.
Note that $\gamma=\delta$ if and only if $\gamma \setminus a =\delta \setminus a$, and so, by induction
$e'_{\gamma \setminus a} e'_{\delta \setminus a} = e'_{\gamma \setminus a}$ if $\gamma=\delta$,
and
$e'_{\gamma \setminus a} e'_{\delta \setminus a} = 0$ if $\gamma\neq \delta$.
Thus we have
$e_\gamma e_\delta=e'_{\gamma \setminus a} a e'_{\delta \setminus a} a=e'_{\gamma \setminus a}  e'_{\delta \setminus a} a$ (by Lemma~\ref{axa} (i)), and this is equal to $e'_{\gamma \setminus a} a=e_\gamma$ if $\gamma=\delta$, and to $0$ if $\gamma\neq \delta$.

-- Suppose now that $a\notin \gamma \cup\delta$. Then $e_\gamma=(1-a)e'_{\gamma}$ and $e_\delta=(1-a)e'_{\delta}$. By Lemma~\ref{axa} (iii), we have $e_\gamma e_\delta=(1-a)e'_{\gamma}(1-a)e'_{\delta}=(1-a)e'_{\gamma}e'_{\delta}$. Thus, $e_\gamma e_\delta$ is $e_\gamma$ if $\gamma=\delta$, and it is $0$ if $\gamma\neq \delta$.

-- Suppose that $a\in \gamma, a\notin \delta$. Then $\gamma\neq \delta$ and $e_\gamma e_\delta= e'_{\gamma \setminus a}a(1-a) e'_{\delta}= e'_{\gamma \setminus a}(a-a^2) e'_{\delta}=0$ since $a$ is idempotent.

-- Suppose that $a\notin \gamma, a\in \delta$. Then $\gamma\neq \delta$ and $e_\gamma e_\delta=(1-a)e'_{\gamma}e'_{\delta \setminus a}a=0$ by Lemma~\ref{axa} (ii).

2. We show that the sum in $\Z\Styl(A)$ of all $e_\gamma$ is equal to 1. Actually we show that this equality holds in the algebra of
noncommutative polynomials. By inspection of~\eqref{egamma}, one sees that this sum is equal to a linear combination of all multilinear
(without repeated letter) words on $A$ of the form $w=xy$, where $x$ is strictly increasing, and $y$ is strictly decreasing. Let $w$ be such a
nonempty word; then $w$ has a unique factorization $w=uzv$, where $u$ is strictly increasing, $v$ is strictly decreasing and $z$ is the
largest letter in $w$. Denote by $U$ the alphabet of $u$, and by $V$ that of $v$. Then the coefficient of $w$ in $e_V$ is $(-1)^{|U|+1}$ and
in $e_{V\cup z}$ it is $(-1)^{|U|}$, while in all other $e_\gamma$ it is 0 (recall that we identify columns in $\Gamma(A)$ and subsets of $A$).
Thus the coefficient of $w$ in the sum is $0$, and therefore the sum is equal to 1.

3. We show that the idempotents are primitive.
First note that since $\Z\Styl(A) \subset \C\Styl(A)$,
it suffices to prove it in $\C\Styl(A)$.
Next, we make use of the following characteristisation: an
idempotent $e$ of a finite dimensional $\C$-algebra $X$ is primitive
if and only if $0$ and $e$ are distinct and are the only idempotents in $e X e$
(see, for instance,
\cite[Section~I.4]{AOR},
\cite[Corollary~4.7]{ASSVol1},
or
\cite[Proposition~A.22]{SteinbergMonoidRepTheory}).
Thus it is enough to prove that
$e_\gamma \C \Styl(A) e_\gamma = \C e_\gamma$,
which we do by induction on the cardinality of $A$.
Let $a = \min(A)$ and $w\in A^*$.

-- Suppose $a \in \gamma$. Then
$e_\gamma w e_\gamma
= e'_{\gamma \setminus a} a w e'_{\gamma \setminus a} a
= e'_{\gamma \setminus a} w' e'_{\gamma \setminus a} a$,
by repeated application of Lemma~\ref{axa} (i), where $w'$ is obtained from $w$ by removing all occurrences of $a$.
Hence, $e'_{\gamma \setminus a} w' e'_{\gamma \setminus a} \in \C\Styl(A \backslash a)$,
so by induction there exists $z \in \C$ such that
$e_\gamma we_\gamma=$
$(e'_{\gamma \setminus a} w' e'_{\gamma \setminus a}) a
= (z e'_{\gamma \setminus a}) a = z e_{\gamma}$.

-- Suppose $a \notin \gamma$. Then
$e_\gamma w e_\gamma
= (1-a) e'_{\gamma \setminus a} w (1-a) e'_{\gamma \setminus a}$.
This is equal to $(1-a) e'_{\gamma \setminus a} w e'_{\gamma \setminus a}$
by Lemma~\ref{axa} (iii), and by induction there exists $z \in \C$ such that
$(1-a) (e'_{\gamma \setminus a} w e'_{\gamma \setminus a})
= (1-a) (z e'_{\gamma \setminus a}) = z e_{\gamma}$.

To conclude, it is enough to show that the $e_\gamma$ are nonzero. For this, it suffices to note that each $e_\gamma$ contains a unique element that is minimal with respect to the $\J$-order on the monoid; we delay the details to the proof of Proposition~\ref{basis} (which will be proved independently), in which we construct a basis of the monoid algebra of $\Styl(A)$ that includes these idempotents.
\end{proof}

\section{The quiver of the stylic algebra}

In this section, we identify the quiver of $\K\Styl(A)$ over a field $\K$.
We do this by defining a quiver $Q(A)$ in \S\ref{Q}
together with a $\K$-algebra morphism $\varphi: \K Q(A) \to \K \Styl(Q)$
in \S\ref{quiver-map} that is surjective (proved in \S\ref{surjectivity-quiver-map}) and
whose kernel is an \emph{admissible ideal} (proved in \S\ref{kernel-quiver-map}).
Such a morphism uniquely determines the quiver of an algebra; see \S\ref{quiver-of-stylic-algebra} for details.
Most of the results hold over $\Z$, so we work over $\Z$ whenever possible.

\subsection{A quiver}\label{Q}

We define a {\it right} action of the monoid $A^*$ on the set $\Gamma(A)$ of columns on $A$. It is enough to define the action of each letter on
each column. Let $c$ be a letter and $\gamma$ a column. If $c<\min(\gamma)$, we let $\gamma\cdot c=\gamma \cup c$.
Otherwise, $c\geq \min(\gamma)$ and we let $b=\max\{x\in\gamma : x\leq c\}$; then $\gamma\cdot c=c\cup (\gamma\setminus b)$; we then say that $b$ is {\it bumped}.
Compactly,
\begin{equation*}
    \gamma \cdot c = \big( \gamma \setminus \max\{x \in \gamma : x \leq c\} \big) \cup \big\{c\big\}.
\end{equation*}

We say that the right action of $c$ on $\gamma$ is {\it frank} if $c\geq \min(\gamma)$ and if $c\notin \gamma$. Note that in this case, $\gamma$ and $\gamma\cdot c$ have the same height.

We define a quiver $Q(A)$ with edges labelled in $A$: its set of vertices is $\Gamma(A)$; and there is a labelled edge $\gamma\xrightarrow{c} \gamma'$ if $\gamma\cdot
c=\gamma'$ and if the action is frank; see Figure \ref{3quivers}.
As usual, the label of a path is the word in $A^*$ that is the product of the labels of the edges of the path.

\begin{figure}
\begin{center}
\begin{tikzpicture}[>=latex,line join=bevel,framed, baseline=-50pt, scale=0.8]
  \node (node_0) at (3.5bp,73.5bp) [draw] {$\epsilon$};
  \node (node_1) at (31.5bp,73.5bp) [draw] {$a$};
  \node (node_2) at (31.5bp,6.5bp) [draw] {$b$};
  \node (node_3) at (64.5bp,73.5bp) [draw] {$ba$};
  \draw [black,->] (node_1) to (node_2);
  \draw (40bp,40.0bp) node [font=\small] {$b$};
\end{tikzpicture}
\qquad
\qquad
\begin{tikzpicture}[>=latex,line join=bevel,framed, scale=0.8]
  \node (node_0) at (-10.5bp,145.5bp) [draw] {$\epsilon$};
  \node (node_1) at (31.5bp,145.5bp) [draw] {$a$};
  \node (node_2) at (9.5bp,76.0bp) [draw] {$b$};
  \node (node_3) at (33.5bp,6.5bp) [draw] {$c$};
  \node (node_4) at (100.5bp,145.5bp) [draw] {$ba$};
  \node (node_5) at (100.5bp,76.0bp) [draw] {$ca$};
  \node (node_6) at (100.5bp,6.5bp) [draw] {$cb$};
  \node (node_7) at (138.5bp,145.5bp) [draw] {$cba$};
  \draw [black,->] (node_1) ..controls (22.792bp,134.24bp) and (18.048bp,127.64bp)  .. (15.5bp,121.0bp) .. controls (12.042bp,111.99bp) and (10.531bp,101.28bp)  .. (node_2);
  \definecolor{strokecol}{rgb}{0.0,0.0,0.0};
  \pgfsetstrokecolor{strokecol}
  \draw (20bp,112.0bp) node [font=\small] {$b$};
  \draw [black,->] (node_1) ..controls (34.839bp,134.07bp) and (36.723bp,127.18bp)  .. (37.5bp,121.0bp) .. controls (42.487bp,81.312bp) and (40.868bp,70.858bp)  .. (37.5bp,31.0bp) .. controls (37.29bp,28.521bp) and (36.959bp,25.909bp)  .. (node_3);
  \draw (48.5bp,76.0bp) node [font=\small] {$c$};
  \draw [black,->] (node_2) ..controls (9.349bp,60.728bp) and (10.106bp,43.953bp)  .. (15.5bp,31.0bp) .. controls (16.967bp,27.477bp) and (19.097bp,24.021bp)  .. (node_3);
  \draw (20.5bp,40.0bp) node [font=\small] {$c$};
  \draw [black,->] (node_4) ..controls (100.5bp,128.68bp) and (100.5bp,107.56bp)  .. (node_5);
  \draw (110.0bp,112.0bp) node [font=\small] {$c$};
  \draw [black,->] (node_5) ..controls (100.5bp,59.182bp) and (100.5bp,38.062bp)  .. (node_6);
  \draw (110.0bp,40.0bp) node [font=\small] {$b$};
\end{tikzpicture}
\end{center}

\begin{center}
\begin{tikzpicture}[>=latex,line join=bevel, scale=0.8, framed]
  \node (node_0) at (-16.084bp,289.5bp) [draw] {$\epsilon$};
  \node (node_1) at (52.084bp,289.5bp) [draw] {$a$};
  \node (node_2) at (10.084bp,220.0bp) [draw] {$b$};
  \node (node_3) at (40.084bp,148.0bp) [draw] {$c$};
  \node (node_4) at (40.084bp,76.0bp) [draw] {$d$};

  \node (node_5) at (197.08bp,289.5bp) [draw] {$ba$};
  \node (node_6) at (174.08bp,220.0bp) [draw] {$ca$};
  \node (node_8) at (198.08bp,148.0bp) [draw] {$da$};
  \node (node_7) at (127.08bp,148.0bp) [draw] {$cb$};
  \node (node_9) at (174.08bp,76.0bp) [draw] {$db$};
  \node (node_10) at (197.08bp,6.5bp) [draw] {$dc$};

  \node (node_11) at (285.08bp,289.5bp) [draw] {$cba$};
  \node (node_12) at (285.08bp,220.0bp) [draw] {$dba$};
  \node (node_13) at (285.08bp,148.0bp) [draw] {$dca$};
  \node (node_14) at (285.08bp,76.0bp) [draw] {$dcb$};
  \node (node_15) at (360.08bp,289.5bp) [draw] {$dcba$};

  \draw [black,->] (node_1) ..controls (37.816bp,283.2bp) and (23.186bp,276.23bp)  .. (16.084bp,265.0bp) .. controls (10.811bp,256.66bp) and (9.3099bp,245.67bp)  .. (node_2);
  \draw (19.084bp,256.0bp) node [font=\small] {$b$};
  \draw [black,->] (node_1) ..controls (49.887bp,262.97bp) and (44.096bp,195.65bp)  .. (node_3);
  \draw (54.084bp,220.0bp) node [font=\small] {$c$};
  \draw [black,->] (node_1) ..controls (58.066bp,278.21bp) and (61.705bp,271.35bp)  .. (64.084bp,265.0bp) .. controls (69.871bp,249.56bp) and (71.069bp,245.37bp)  .. (73.084bp,229.0bp) .. controls (77.966bp,189.34bp) and (85.587bp,146.09bp)  .. (66.084bp,103.0bp) .. controls (63.451bp,97.184bp) and (58.908bp,91.945bp)  .. (node_4);
  \draw (86.084bp,184.0bp) node [font=\small] {$d$};
  \draw [black,->] (node_2) ..controls (9.5925bp,204.58bp) and (10.101bp,187.69bp)  .. (16.084bp,175.0bp) .. controls (18.51bp,169.85bp) and (22.322bp,165.03bp)  .. (node_3);
  \draw (21.084bp,184.0bp) node [font=\small] {$c$};
  \draw [black,->] (node_2) ..controls (4.2842bp,200.26bp) and (-4.4016bp,166.15bp)  .. (3.0837bp,139.0bp) .. controls (8.0786bp,120.88bp) and (19.913bp,102.72bp)  .. (node_4);
  \draw (9.084bp,148.0bp) node [font=\small] {$d$};
  \draw [black,->] (node_3) ..controls (40.084bp,130.49bp) and (40.084bp,108.25bp)  .. (node_4);
  \draw (47.084bp,112.0bp) node [font=\small] {$d$};

  \draw [black,->] (node_5) ..controls (186.58bp,278.56bp) and (180.92bp,272.0bp)  .. (178.08bp,265.0bp) .. controls (174.48bp,256.11bp) and (173.44bp,245.39bp)  .. (node_6);
  \draw (181.58bp,256.0bp) node [font=\small] {$c$};
  \draw [black,->] (node_5) ..controls (201.53bp,278.16bp) and (204.04bp,271.29bp)  .. (205.08bp,265.0bp) .. controls (211.62bp,225.54bp) and (210.33bp,214.65bp)  .. (205.08bp,175.0bp) .. controls (204.63bp,171.59bp) and (203.86bp,167.99bp)  .. (node_8);
  \draw (218.58bp,220.0bp) node [font=\small] {$d$};
  \draw [black,->] (node_6) ..controls (158.48bp,208.88bp) and (149.05bp,201.52bp)  .. (143.08bp,193.0bp) .. controls (137.05bp,184.39bp) and (133.03bp,173.24bp)  .. (node_7);
  \draw (147.0bp,184.0bp) node [font=\small] {$b$};
  \draw [black,->] (node_6) ..controls (173.12bp,204.75bp) and (172.99bp,188.0bp)  .. (178.08bp,175.0bp) .. controls (179.81bp,170.59bp) and (182.53bp,166.29bp)  .. (node_8);
  \draw (183.58bp,184.0bp) node [font=\small] {$d$};
  \draw [black,->] (node_7) ..controls (129.99bp,132.62bp) and (134.51bp,115.25bp)  .. (143.08bp,103.0bp) .. controls (146.9bp,97.544bp) and (152.15bp,92.564bp)  .. (node_9);
  \draw (147.58bp,112.0bp) node [font=\small] {$d$};
  \draw [black,->] (node_8) ..controls (187.36bp,136.31bp) and (181.16bp,128.85bp)  .. (178.08bp,121.0bp) .. controls (174.58bp,112.06bp) and (173.55bp,101.35bp)  .. (node_9);
  \draw (183.58bp,112.0bp) node [font=\small] {$b$};
  \draw [black,->] (node_8) ..controls (201.86bp,135.89bp) and (204.15bp,128.04bp)  .. (205.08bp,121.0bp) .. controls (210.33bp,81.345bp) and (211.62bp,70.463bp)  .. (205.08bp,31.0bp) .. controls (204.63bp,28.247bp) and (203.89bp,25.383bp)  .. (node_10);
  \draw (218.58bp,76.0bp) node [font=\small] {$c$};
  \draw [black,->] (node_9) ..controls (173.03bp,60.718bp) and (172.84bp,43.937bp)  .. (178.08bp,31.0bp) .. controls (179.55bp,27.39bp) and (181.76bp,23.897bp)  .. (node_10);
  \draw (184.58bp,40.0bp) node [font=\small] {$c$};

  \draw [black,->] (node_11) ..controls (285.08bp,272.68bp) and (285.08bp,251.56bp)  .. (node_12);
  \draw (296.08bp,256.0bp) node [font=\small] {$d$};
  \draw [black,->] (node_12) ..controls (285.08bp,202.49bp) and (285.08bp,180.25bp)  .. (node_13);
  \draw (296.08bp,184.0bp) node [font=\small] {$c$};
  \draw [black,->] (node_13) ..controls (285.08bp,130.49bp) and (285.08bp,108.25bp)  .. (node_14);
  \draw (296.08bp,112.0bp) node [font=\small] {$b$};

\end{tikzpicture}
\end{center}
\caption{The quivers for alphabets of cardinality $2,3,4$; the columns are
represented by decreasing words and the empty word is denoted $\epsilon$.}
\label{3quivers}
\end{figure}
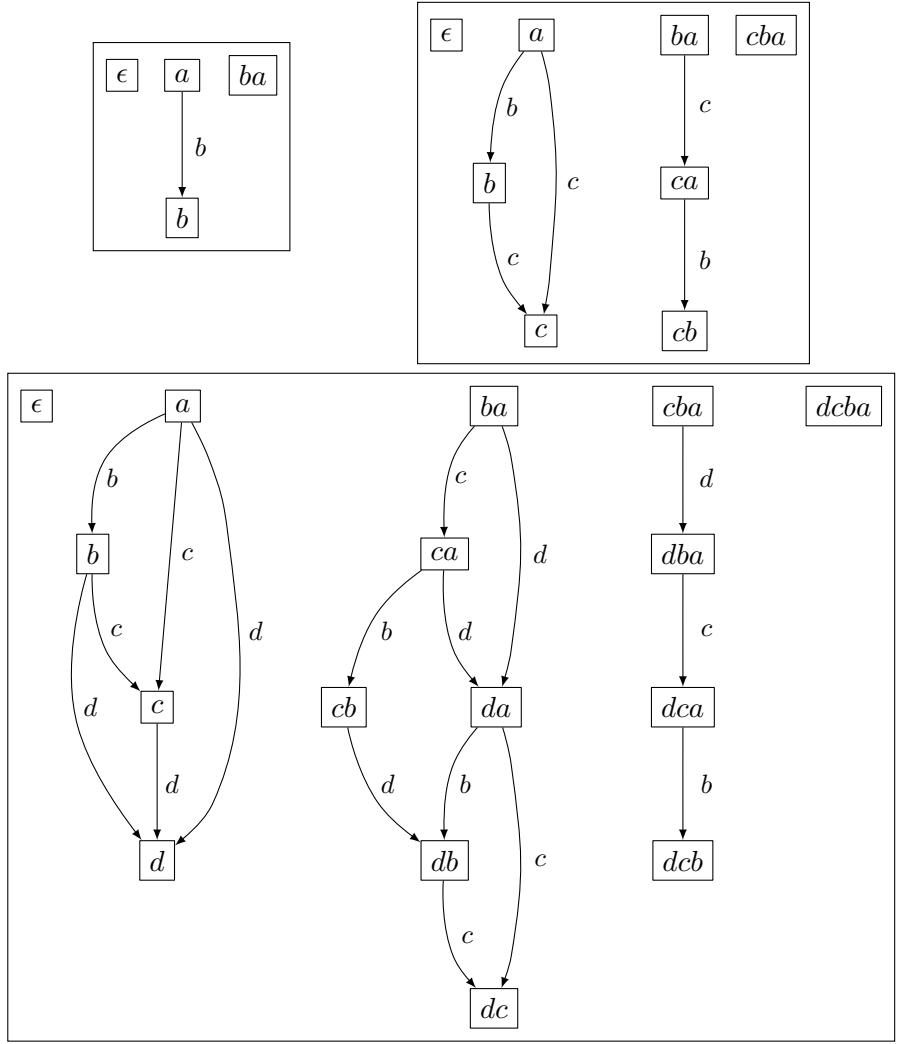

For later use, we note the following result relating the left and right actions. The proof is left to the reader.

\begin{lemma}\label{left-right-action}
For two columns of the same height $\gamma,\delta$, and two letters $b,c$, the two following conditions are equivalent:

(i) $b\cdot \delta=\gamma$ and $c$ is bumped;

(ii) $\gamma \cdot c=\delta$, and $b$ is bumped.
\end{lemma}

\subsection{A lemma on edges and idempotents}

We give a technical, but important, result on the idempotents of the stylic algebra and the quiver introduced previously.

\begin{lemma}\label{identity} Let $\gamma\xrightarrow{c} \delta$ be an edge in the quiver $Q(A)$, and denote by $b$ the bumped letter, so that
$\delta=c\cup (\gamma\setminus b)$, and $b\in \gamma, c\notin\gamma, b\notin \delta, c\in \delta$. Then in $\Z\Styl(A)$
\begin{equation*}
    be_\gamma c=bce_\delta
    \qquad\text{and}\qquad
    e_\gamma c e_\delta=e_\gamma c.
\end{equation*}
\end{lemma}

\begin{proof} I. We prove the first identity. Let $a=\min(A)$. As in the proof of Theorem \ref{system}, denote by $e'_{\gamma'}$ the idempotents (\ref{egamma}) relative to the alphabet $A\setminus a$.

1. Suppose that $a\in \gamma\cap \delta$. Since the action $\gamma\cdot c$ is frank, and since $a$ cannot be bumped, we have $a<b<c$. Let $\gamma'=\gamma \setminus a$ and $\delta'=\delta\setminus a$. Then $\delta'=\gamma'\cdot c$ and the action is frank. By induction, we deduce that
$be'_{\gamma'} c=bce'_{\delta'}$. Note that the minimum of $\gamma'$ is a letter $x$ such that $a<x<c$; thus $xac=xca$ and since $x$ is the last factor in the product (\ref{egamma}) defining $e'_{\gamma'}$, we have $e'_{\gamma'}ac=e'_{\gamma'}ca$. We have $e_\gamma=e'_{\gamma'}a$ and $e_\delta=e'_{\delta'}a$.
Thus
$be_\gamma c= be'_{\gamma'}ac=be'_{\gamma'}ca=bce'_{\delta'}a=bce_{\delta}$.

2. Suppose that $a\notin \gamma\cup \delta$. Then $a<b<c$. Moreover $\delta \xrightarrow{c} \delta'$ is an edge in the quiver $Q(A\setminus a)$ and $b$ is bumped. With notations similar to 1, we have $e_\gamma = (1-a)e'_{\gamma'}$ and $e_\delta=(1-a)e'_{\delta'}$. Since $b(1-a)b=b^2-bab=b^2-ba=b(1-a)$ and $bac=bca$, we have
$be_\gamma c= b(1-a)e'_{\gamma'}c
=b(1-a)be'_{\gamma'}c=b(1-a)bce'_{\delta'}$  (by induction) $=b(1-a)ce'_{\delta'}=bc(1-a)e'_{\delta'}=bce_\delta$.

3. Suppose that $a\in\gamma$ and $a\notin \delta$. Then the bumped letter is $b=a$. We denote by $\gamma$ and $\delta$ the decreasing words associated with these two columns. Since $c\in\delta$, we have $\delta=\delta_1 c\delta_2$, where each letter in $\delta_1$ is larger than $c$; hence
$c\delta\equiv_{styl}\delta$ by Lemma \ref{axa} (i). Moreover, $\delta c\equiv_{styl}\delta$ since $c$ is the smallest, hence last, letter of $\delta=\delta'c$ and $c^2 = c$.

We have \begin{equation}\label{gammac}
\gamma c\equiv_ {styl}a\delta,
\end{equation}
since this holds even plactically as one sees by computing the image under $P$ of both sides (for $P(a\delta)$, Schensted left insert $a$ into $\delta$ and use Lemma~\ref{left-right-action}).

Let
    $y_1, \ldots, y_s, z_1, \ldots, z_t$
    be the letters in $A$ that do not appear in $\gamma$ nor in $\delta$,
    ordered so that
    \begin{equation*}
        a ~<~
        y_1 < \cdots < y_s
        ~<~ c ~<~
        z_1 < \cdots < z_t.
    \end{equation*}
Then
\begin{align*}
    e_\gamma &= \prod_{j=1}^{s}(1 - y_j) (1 - c) \prod_{k=1}^{t}(1 - z_k) \gamma
    \\
    e_\delta &= (1 - a) \prod_{j=1}^{s}(1 - y_j) \prod_{k=1}^{t}(1 - z_k) \delta.
\end{align*}
Thus, by (\ref{gammac}), $ae_\gamma c=a\prod_{j=1}^{s}(1 - y_j) (1 - c) \prod_{k=1}^{t}(1 - z_k)a\delta$.
Note that $\prod_{j=1}^{s}(1 - y_j) $ is equal to $1$ plus a linear combination of $uy$, with $$a<y<c<z_1<\cdots <z_t. $$
Therefore, $ae_\gamma c$ is equal to $a (1 - c) \prod_{k=1}^{t}(1 - z_k)a\delta$ plus a linear combination of $auy (1 - c) \prod_{k=1}^{t}(1 - z_k)a\delta$, and we show that each term in the linear combination vanishes.

Note that it suffices to show that $y \prod_{k=1}^{t}(1 - z_k)a\delta = y c \prod_{k=1}^{t}(1 - z_k)a\delta$.
We prove this equality, starting from the right-hand side: since $\delta=\delta'c$, we have
$yc\prod_{k=1}^{t}
(1 - z_k)a\delta=yc\prod_{k=1}^{t}
(1 - z_k)a\delta' c=yac\prod_{k=1}^{t}(1 - z_k)\delta' c$ (since by the second identity in Lemma \ref{superplax}, we have $
yac\prod_{k=1}^{t}(1 - z_k)=yc\prod_{k=1}^{t}
(1 - z_k)a$) $=ya\prod_{k=1}^{t}(1 -
z_k)\delta' c$ (by Lemma~\ref{axa}~(i), since all letters $z_i$ and in $\delta'$ are $>c$)
$=y\prod_{k=1}^{t}(1 -
z_k)a\delta$, by the same identity in Lemma \ref{superplax}.

 It follows that $ae_\gamma c=a(1-c)\prod_{k=1}^{t}(1 - z_k)a\delta=a(1-c)\prod_{k=1}^{t}(1 - z_k)ac\delta
 $ (since $c\delta=\delta$) $
 =a(1-c)a\prod_{k=1}^{t}(1 - z_k)c\delta$ (by the first identity in Lemma \ref{superplax}) $=a(1-c)ac\prod_{k=1}^{t}(1 - z_k)c\delta$
 (by Lemma~\ref{axa}~(i))
$ =(ac-ca)\prod_{k=1}^{t}(1 - z_k)\delta$.

On the other hand, we have $ace_\delta=ac(1 - a) \prod_{j=1}^{s}(1 - y_j) \prod_{k=1}^{t}(1 - z_k) \delta=(ac-ca)\prod_{j=1}^{s}(1 - y_j) \prod_{k=1}^{t}(1 -
z_k) \delta$. Note that $\prod_{j=1}^{s}(1 - y_j) $ is equal to $1$ plus a linear combination of $yu$, with $a<y<c$. Since $(ac-ca)y=acy-cay=0$ (plactic relation), we obtain $ace_\delta=(ac-ca)\prod_{k=1}^{t}(1 - z_k)\delta$.

It follows that $ae_\gamma c=ace_\delta$.

4. The last case to consider is when $a \notin \gamma$ and $a \in \delta$; however, it
does not occur because the action $\gamma\cdot c$ is frank (in particular, if
$a \notin \gamma$, then $a \notin \delta$).

II. We prove now the second identity. Note that $\gamma=\gamma_1 b\gamma_2$, where each letter in $\gamma_2$ is smaller that $b$; hence $\gamma b=\gamma$, by the dual statement of Lemma \ref{axa} (i). We have, using the fact that $\gamma$ is idempotent in $\Styl(A)$ (see the sentence before (\ref{egammagamma})):
    \begin{align*}
        e_{\gamma} c e_{\delta}
        & = e_{\gamma} b c e_{\delta}
        & \text{(since \small$e_{\gamma} = e_\gamma \gamma= e_\gamma \gamma b = e_\gamma b$)} \\
        & = e_{\gamma} b  e_{\gamma} c
        & \text{\small(by the first identity in the lemma, already proved)} \\
        & =  e_{\gamma}  e_{\gamma}c
        & \text{\small(since $e_\gamma b=e_\gamma$)} \\
        & = e_\gamma c
        & \text{\small(since $e_\gamma$ is idempotent)}.
        & \qedhere
    \end{align*}
\end{proof}

\subsection{A quiver map}
\label{quiver-map}

Let $Q=Q(A)$ be the quiver defined in Subsection \ref{Q}.
The {\it path algebra} $\Z Q$ is the free $\Z$-module with basis the set of
paths in the quiver, including an empty path around each vertex $\gamma$ (this empty path is denoted $\gamma$); the product is
the unique product extending the natural product of paths.

We define a $\Z$-linear mapping $\varphi:\Z Q\to \Z\Styl(A)$ as follows:
\begin{itemize}
    \item
        if $\gamma$ is an empty path, then
        \begin{equation*}
            \varphi(\gamma) = e_{\gamma};
        \end{equation*}

    \item
        if
        \begin{equation}\label{path}
        \gamma_0 \xrightarrow{c_1} \gamma_1 \xrightarrow{c_2} \cdots \xrightarrow{c_l} \gamma_l
        \end{equation}
        is a path in $Q$, then its image under $\varphi$ is
        $$
        e_{\gamma_0} c_1 e_{\gamma_1} c_2 \cdots c_l e_{\gamma_l}.
        $$
\end{itemize}
Note that this mapping is a $\Z$-algebra homomorphism.

\begin{theorem}\label{image} The image under $\varphi$ of a path from $\gamma$ to $\delta$ with label $u$ is $e_\gamma u$.
\end{theorem}

\begin{proof} This is clear if the path is of length $0$. Suppose it is true for each path of length $l\geq 0$. Consider a path
$\gamma_0 \xrightarrow{c_1} \gamma_1 \xrightarrow{c_2} \cdots \xrightarrow{c_{l+1}} \gamma_{l+1}$. Its image under $\varphi$ is by definition $x=\varphi(p)e_{\gamma_{l}}c_{l+1}e_{\gamma_{l+1}}$, where $p$ is the path (\ref{path}). Thus $x=\varphi(p)e_{\gamma_{l}}c_{l+1}$  (by the second equality in Lemma \ref{identity})
$=\varphi(p)c_{l+1}$ (since $\varphi (p)e_{\gamma_l}=\varphi(p)$ by definition of $\varphi$ and the idempotence of  $e_{\gamma_l}$)
$=e_{\gamma_0}c_1\cdots c_lc_{l+1}$ (by induction).
\end{proof}

\begin{corollary}\label{same-image} Consider two paths in $Q(A)$ starting from the same vertex $\gamma$, with labels $u,v$. If $\gamma u\equiv_{styl} \gamma v$, then these
paths have the same image under $\varphi$.
\end{corollary}

\begin{proof} The images of these paths are $e_\gamma u$ and $e_\gamma v$, respectively. These elements are by (\ref{egammagamma}) equal to $e_\gamma \gamma u$ and $e_\gamma \gamma v$. Thus,
the lemma follows.
\end{proof}

\subsection{Extended quiver}
\label{extended-quiver}

The {\it extended quiver} $Q'(A)$ has the same set of vertices as $Q(A)$, has all edges of $Q(A)$, together with new edges, which are loops: for
each column $\gamma$ and each $c\in\gamma$, we have in $Q'(A)$ the edge $$\gamma\xrightarrow{c}\gamma.$$ It is clearly a deterministic automaton. Note that if $c\in A$ and $\gamma \in \Gamma(A)$, there is an edge labelled $c$ starting from $\gamma$ in $Q'(A)$ if and only if $c\geq min(\gamma)$. Moreover, if for $\gamma,\delta\in \Gamma(A), w\in A^*$, there is a path $\gamma \xrightarrow{w}\delta$ in $Q'(A)$, then $\delta=\gamma\cdot w$.

\begin{proposition}\label{extended-path} Let $x\in \Styl(A)$ and denote by $\gamma w$ the column-reading word of the
$N$-tableau $N(x)$ of $x$, with $\gamma$ being the first column of $N(x)$. Then there is a unique path in the extended quiver, starting
from $\gamma$, with label $w$.
\end{proposition}

Before proving the proposition, we prove a useful lemma, showing that the involution $\theta$ defined in Section \ref{theta} conjugates the left and right actions.

\begin{lemma}\label{left-right}
Let $w\in A^*$  and $\gamma\in \Gamma(A)$. Then
$$\gamma\cdot w=\theta(\theta(w)\cdot\theta(\gamma)).$$
\end{lemma}

\begin{proof} For $w\in A$, the formula follows from the definitions of the left and right actions on columns. To conclude, it is enough
to prove that if the formula holds for $u,v\in A^*$, then also for $w=uv$. We have $\gamma\cdot w=\gamma\cdot
(uv)=(\gamma\cdot u)\cdot v=\theta(\theta(v)\cdot\theta(\gamma\cdot u))=\theta(\theta(v)\cdot(\theta(u)\cdot\theta(\gamma))
=\theta((\theta(v)\theta(u))\cdot\theta(\gamma))=
\theta(\theta(uv)\cdot\theta(\gamma))
=\theta(\theta(w)\cdot\theta(\gamma)).
$
\end{proof}

\begin{proof}[Proof of Proposition \ref{extended-path}] Uniqueness follows from the deterministic property of $Q'(A)$ viewed as an automaton.

To prove the existence of this path, it is enough, by the definition of the right action and of the extended quiver, to show that the
height of $\gamma\cdot p$ is equal to the height $k$ of $\gamma$, for each prefix $p$ of $w$.

Since $\gamma w$ is the column-reading word of $N(x)$, it follows from \eqref{Pw=T} that the $P$-tableau of $\gamma w$ is equal to $N(x)$.
Thus, by Schensted's theorem, the height $k$ of $N(x)$ is equal to the length of the longest strictly
decreasing subsequence of $\gamma w$. Now, the length of the longest strictly decreasing
subsequence of $\theta(\gamma w)$ is $k$, too. Hence, the height of the $P$-tableau of $\theta(\gamma w)$ is $k$; by the definition
of left action, the first column of this tableau is $\theta(\gamma w)\cdot\emptyset$, and this column is equal to $(\theta(w)
\theta(\gamma))
\cdot\emptyset=\theta(w)\cdot(\theta(\gamma)\cdot\emptyset)=\theta(w)\cdot\theta(\gamma)$. Therefore, applying $\theta$ and using Lemma \ref{left-right}, we see that
$\gamma\cdot w$ is of height $k$.

Since the (left and right) action on columns never decreases the height, it follows that for each prefix $p$ of $w$, the
height of $\gamma\cdot p$ is equal to $k$.
\end{proof}

Lemma \ref{left-right} has the following corollary.

\begin{corollary}\label{gamma-w}
Let $w$ be a word and $\gamma $ be a column. Then $\gamma w\equiv_{styl} u (\gamma\cdot w)$ for some word $u$.
\end{corollary}

\begin{proof}
We know that $w\cdot \gamma$ is the first column of $P(w\gamma)$. It
follows by column reading and Schensted's theorem that $w\gamma \equiv_{plax} (w\cdot \gamma) u$ for some word $u$.
Applying $\theta$ and using Lemma \ref{left-right}, we find that for each word $w$ and each column $\gamma$, $\gamma w\equiv_{plax} u (\gamma\cdot w)$ for
some word $u$; therefore $\gamma w\equiv_{styl} u (\gamma\cdot w)$.
\end{proof}

Each path
\begin{equation}\label{pathQ'}
\gamma \xrightarrow{w} \gamma\cdot w
\end{equation}
in the extended quiver $Q'(A)$, starting from vertex $\gamma$ and with label $w$, defines a path
\begin{equation}\label{pathQ}
\gamma \xrightarrow{w'} \gamma\cdot w
\end{equation}
in the quiver $Q(A)$, by removing the loops. Precisely, we define the label $w'$ of the associated path in $Q(A)$ recursively as follows: if $w$ is
empty, $w'=w$; otherwise $w=uc$, $u\in A^*, c\in A$, $u'$ is constructed by induction, and then:
\begin{itemize}
    \item[--]
        first case: $w'=u'$ if $\gamma\cdot u=(\gamma \cdot u)\cdot c$ (equivalently
        $c\in \gamma\cdot u$);

    \item[--]
        second case: $w'=u'c$ otherwise.
\end{itemize}
We call this construction {\it loops removal}. 

\begin{lemma}\label{same-action}
With these notations, $\gamma\cdot w'=\gamma\cdot w$.
\end{lemma}
 
\begin{proof} We follow the construction. If $w$ is empty, then $w'$ is empty, and the equality is evident. Suppose now that $w=uc$. In the first case, $\gamma\cdot w'=\gamma\cdot  u'=\gamma\cdot u$ (by induction) $=\gamma\cdot(uc)=\gamma\cdot w$. In the second case, $\gamma\cdot w'=\gamma\cdot(u'c)=(\gamma\cdot u')\cdot c=(\gamma\cdot u)\cdot c$ (by induction) $=\gamma\cdot (uc)=\gamma\cdot w$.
\end{proof}

\begin{lemma}\label{gammaw} With these notations, $\gamma w'\equiv_{styl}\gamma w$.
\end{lemma}

\begin{proof} 1. Let $\delta\xrightarrow{c} \delta\cdot c$ be a 
an edge in the extended quiver $Q'(A)$. Then $c\geq \min (\delta)$. Next, $\delta c\equiv_{styl}b (\delta\cdot c)$, with $b\in
A$: this equality holds indeed plactically, as a particular case of the presentation by columns of the plactic monoid due to \cite{BCCL,
CGM}, after applying $\theta$ and Lemma \ref{left-right} (see also \cite{AR} Proposition 12.3 (v)).

Suppose that moreover $c\in\delta$, equivalently $\delta\cdot c=\delta$. Then $\delta=\delta_1c\delta_2$ with each letter in $\delta_2$ 
smaller than $c$. Then $\delta_2 c\equiv_{styl} \delta_2$, by the dual form of Lemma \ref{axa} (i), from which follows $\delta c\equiv_{styl} \delta$.

%
%

2. We prove the lemma by following the recursive construction of $w'$. If $w$ is empty, it is evident. Suppose now that $w=uc$
and assume by induction that $\gamma u\equiv_{styl}\gamma u'$, where $u'$ is obtained from $u$ by loops removal. By Corollary \ref{gamma-w}, we have $\gamma u\equiv_{styl}v (\gamma\cdot u)$ for some word $v$. By
1 and 2, we have $(\gamma\cdot u) c\equiv_{styl} b(\gamma\cdot w)$,
with $b=1$ if $\gamma\cdot u=\gamma \cdot(uc)$, and $b\in A$ otherwise.

In the first case, we have  $\gamma\cdot u=\gamma\cdot (uc)=\gamma \cdot w$, $w'=u'$, $b=1$.
Then $\gamma w=\gamma uc \equiv_{styl} v(\gamma\cdot u)c \equiv_{styl} v(\gamma\cdot w)$, and $\gamma w'=\gamma u'
\equiv_{styl}  \gamma u \equiv_{styl}
v(\gamma\cdot u)=v(\gamma \cdot w)$.

In the second case, we have $\gamma\cdot w=\gamma\cdot (uc) \neq \gamma\cdot u$, $w'=u'c$, $b\in A$. Then $\gamma w=\gamma uc \equiv_{styl}
v(\gamma\cdot u)c \equiv_{styl} vb(\gamma\cdot w)$, and $\gamma w'=\gamma u'c \equiv_{styl} \gamma uc
\equiv_{styl} v(\gamma\cdot u)c
\equiv_{styl} vb(\gamma\cdot w)$.
\end{proof}

\begin{corollary}\label{path-quiver} The image under $\varphi$ of the path \eqref{pathQ} in $Q(A)$, obtained from the path \eqref{pathQ'} in $Q'(A)$ by loops removal, is equal to $e_\gamma w$.
\end{corollary}

\begin{proof}
Suppose that $w$ is the label of a path in $Q'(A)$ starting form $\gamma$; define $w'$ by loops removal. The image under $\varphi$ of our path of $Q(A)$ is by Theorem \ref{image} equal to $y=e_\gamma w'$. By (\ref{egammagamma}), we have $y=e_\gamma \gamma w'$. Hence by Lemma \ref{gammaw}, the corollary follows.
\end{proof}

\subsection{The surjectivity of the quiver map}
\label{surjectivity-quiver-map}

Let $x \in \Styl(A)$ and denote by $\eta(x) w_x$ the column-reading word of the
$N$-tableau $N(x)$ of $x$, with $\eta(x)$ being the first column of $N(x)$.
Recall from Proposition~\ref{extended-path}, that we have constructed a path, in the extended quiver $Q'(A)$, starting form $\eta(x)$
and with label $w_x$. From this path in $Q'(A)$, we obtain by loops removal in Section \ref{extended-quiver}, 
a path in $Q(A)$ starting form $\eta(x)$
and with label $w'_x$; we call such
a path an {\it $N$-path}.

\begin{proposition}\label{basis} The set $\{e_{\eta(x)}w_x : x\in\Styl(A)\}$ is a basis of $\Z\Styl(A)$.
\end{proposition}

\begin{proof} Recall from \cite{AR} that $\Styl(A)$ is a $\J$-trivial monoid, and that it has therefore the $\J$-order $\leq_\J$. One has $x\leq_\J y$ if and only if for some $u,v$, $x=uyv$ (all these elements are in $\Styl(A)$).

Let $x\in\Styl(A)$, with $\eta(x)=\gamma$; then by (\ref{xwN}), $x=\gamma w_x$ in $\Z\Styl(A)$, and by (\ref{egamma}),
$$e_{\eta(x)}w_x=\prod^{\nearrow}_{a\notin \gamma}(1-a)x.$$
Let $a\notin \gamma$; then $ax\leq_\J x$; moreover by Proposition \ref{properties}, and since $\gamma$ is the first column of $N(x)$, $ax\neq x$
and therefore $ax<_\J x$. It follows from the displayed formula that $e_{\eta(x)}w_x$ is equal to $x$ plus a linear combination of elements strictly smaller than $x $ in the $\J$-order.
Hence, by triangularity, the elements $e_{\eta(x)}w_x$, $x\in\Styl(A)$, form a basis of $\Z\Styl(A)$.
\end{proof}

\begin{corollary}
    \label{quiver-map-is-surjective}
    The quiver map $\varphi$ is surjective.
\end{corollary}

\begin{proof} The element $e_{\eta(x)}w_x$ is the image under $\varphi$ of the path constructed in Corollary~\ref{path-quiver}. Hence, by Proposition~\ref{basis}, $\varphi$ is surjective.
\end{proof}

\begin{corollary}\label{lin-ind}
    The $N$-paths are linearly independent modulo $\ker(\varphi)$.
\end{corollary}

\subsection{The kernel of the quiver map}
\label{kernel-quiver-map}

The following result shows that $\ker(\varphi)$ is completely described by Corollary~\ref{same-image}.

\begin{proposition}\label{span} The kernel of $\varphi$ is spanned by the elements which are differences of two paths in $Q(A)$ starting from the same vertex $\gamma$ and having labels $u,v$ satisfying $\gamma u\equiv_{styl} \gamma v$.
\end{proposition}

\begin{proof} Denote by $H$ the subspace described in the statement. We know by Corollary~\ref{same-image} that $H$ is a subspace of $\ker(\varphi)$.

Consider a path starting from $\gamma$ and with label $u$. Let $x=\gamma u$.

1. We show that the first column of $N(x)$ is $\gamma$.

If $a\in \gamma$, by Lemma~\ref{axa} (i), we have $a\gamma=\gamma$ in $\Styl(A)$; hence, $a$ is in the first column of $N(x)$ (by Proposition~\ref{properties}), and this column therefore contains $\gamma$.

Moreover, by definition of the quiver and of paths, the height of $\gamma\cdot u$ is the same as the height $h$ of $\gamma$; thus the height of $\theta(\gamma\cdot u)$ is $h$, and so is that of $\theta(u)\cdot \theta(\gamma)$ by Lemma~\ref{left-right}; but this column is the first column of $N(\theta(u)\theta(\gamma))=N(\theta(\gamma u))=N(\theta(x))$. By Theorem 9.1 in \cite{AR}, $N(x)$ and $N(\theta(x))$ have the same height; the height of $N(x)$ is therefore $h$. It follows that its first column is $\gamma$.

2. Consider now the path \eqref{pathQ'} of $Q'(A)$ constructed in Proposition~\ref{extended-path}, and the associated path \eqref{pathQ} in $Q(A)$, obtained by removing the loops: it starts at $\gamma$ and has $w'$ as label.

We know that $\gamma u=x\equiv_{styl} \gamma w$, by \eqref{xwN} since the latter word is the column-reading word of $N(x)$. Hence by Lemma~\ref{gammaw}, $\gamma u\equiv_{styl} \gamma w'$ and therefore the two paths of $Q(A)$ starting at $\gamma$ and with labels $u$ and $w'$ have the same image under $\varphi$, by Corollary~\ref{same-image}.

It follows that each element in the quiver algebra is congruent modulo $H$ to a linear combination of $N$-paths. Since by
Corollary~\ref{lin-ind} these $N$-paths are linearly independant modulo $\ker(\varphi)$, it follows that $\ker(\varphi)\subset H$.
\end{proof}

\subsection{The quiver of the stylic algebra}
\label{quiver-of-stylic-algebra}

Now, let $\K$ be a field of characteristic $0$. We apply a theorem of Auslander, Reiten and Smal\o {} \cite{AOR}, in order to prove that $Q(A)$ is the quiver of $\K\Styl(A)$.

\begin{theorem}\label{quiver} The quiver of the stylic algebra over $\K$ is $Q(A)$.
\end{theorem}

We first prove the following useful lemma.
\begin{lemma}\label{gammau} Let $\gamma$ be a column, and $u,v\in A^*$. If $\gamma u\equiv_{styl} \gamma v$, then $\gamma \cdot u =\gamma \cdot v$.
\end{lemma}

\begin{proof}
By Lemma \ref{left-right}, it is enough to prove the dual statement: if $u\gamma \equiv_{styl} v\gamma $, then
$u\cdot\gamma=v\cdot \gamma $. The hypothesis implies that $(u\gamma )\cdot\emptyset=(v\gamma )\cdot\emptyset$.
This implies $u\cdot(\gamma\cdot\emptyset)=v\cdot(\gamma \cdot \emptyset)$, thus $u\cdot\gamma=v\cdot \gamma$.
\end{proof}

\begin{proof}[Proof of Theorem \ref{quiver}] According to a theorem in  \cite{AOR}, in the formulation of \cite[Theorem 3.3.4]{DHST}, it is enough to show that the ideal $\ker(\varphi)$ is {\it admissible}. This means that $F^m \subset \ker(\varphi) \subset F^2$, where $F$ is the ideal in $\K Q(A)$ generated by the arrows of $Q(A)$.

The first inclusion is clear, since the quiver has no closed path, so that for $m$ large enough, $F^m=0$.

We know that $\ker(\varphi)$ is spanned by the elements, differences of two paths, described in Proposition \ref{span}, whose
notations we use now. In particular, $\gamma u\equiv_{styl} \gamma v$. Thus it is enough to show that $u,v$ are both of length at least 2. We may assume that the element is nonzero.

Observation 1: the alphabet of $\gamma u$ and $\gamma v$ must be equal, since these words are stylically congruent.

Observation 2: assuming that the alphabet is $1,2,3,\ldots$, call {\it weight} of a column the sum of its elements. Then by definition of frank action, the weight of $\gamma \cdot a$ is larger that the weight of $\gamma$, and so the weight of the vertices strictly increases along a path in $Q(A)$.

Suppose by contradiction that $u$ is of length $0$ or $1$,
and we begin by  length 0.
If $v$ also is of length 0, the element is 0, which was excluded. If $v$ is of positive length then, since the action is frank, the alphabet of $\gamma v$ is strictly larger than that of $\gamma$;
hence
the alphabets of $\gamma u=\gamma$ and of $\gamma v$ differ, so that by Observation 1, we cannot have $\gamma u\equiv_{styl} \gamma v$.

Thus we may assume that $u=b$ is of length 1. Then $v$ cannot be of length 0, by the same argument just given. If $v$ is of length
1, then by Observation 1, and since the two actions are frank (so that $u,v\notin \gamma$), we must have $u=v$, and the element
is 0, which was excluded. Thus $v$ is of length at least 2: $v=cv'$, $v'$ nonempty; then by Observation 1, $c$ appears in
$\gamma b$, but not in $\gamma$, since the action $\gamma \cdot c$ is frank, hence $c=b$; but then by Observation 2, the weight of
$\gamma \cdot v$ is
larger than that of $\gamma \cdot b$, and we cannot have the equality $\gamma \cdot u=\gamma \cdot v$, contradicting $\gamma u\equiv_{styl} \gamma v$ by Lemma \ref{gammau}.
\end{proof}

\subsection{Cartan invariants and Indecomposable Projective Modules}

The \emph{Cartan invariants} of a finite dimensional $\K$-algebra $\Lambda$ are
the numbers $\dim_\K(e_i \Lambda e_j)$, where $\{e_1, \ldots, e_n\}$ is a complete
system of primitive orthogonal idempotents of $\Lambda$. They do not depend on the
choice of the complete system.

In the case of the stylic monoid, we are therefore interested in computing the dimension
of the subspaces $e_{\gamma} \K \Styl(A) e_{\gamma'}$ for $\gamma, \gamma' \in \Gamma(A)$.

\begin{proposition}
    \label{cartan-invariants}
    For $\gamma, \gamma' \in \Gamma(A)$,
    \begin{equation*}
        \dim \Big( e_{\gamma} \K \Styl(A) e_{\gamma'} \Big) =
        \Big|
        \big\{
            x \in \Styl(A) : \eta(x) = \gamma
            \text{~and~}
            \theta(\eta(\theta(x)) = \gamma'
        \big\}
        \Big|.
    \end{equation*}
\end{proposition}

\begin{proof}
    By Proposition~\ref{basis}, we have that
    $\{e_{\eta(x)}w_x : x\in\Styl(A)\}$
    is a basis of $\Z \Styl(A)$.
    Moreover, each $e_{\eta(x)}w_x$ is the image under $\varphi$ of a path in
    $Q'(A)$ that starts at $\eta(x)$, is labelled $w_x$, and ends at
    $\theta(\eta(\theta(x)))$; see \S\ref{surjectivity-quiver-map}.
    Thus,
    \begin{equation*}
        e_{\eta(x)} w_x = e_{\eta(x)} w_x e_{\theta(\eta(\theta(x)))}.
    \end{equation*}
    It follows that
    $\{e_{\eta(x)} w_x e_{\theta(\eta(\theta(x)))} : x\in\Styl(A) \}$
    is a basis of $\Z \Styl(A)$,
    and that
    \begin{equation*}
        \{e_{\eta(x)} w_x e_{\theta(\eta(\theta(x)))} : x\in\Styl(A)
        \text{~with $\eta(x) = \gamma$ and $\theta(\eta(\theta(x))) = \gamma'$} \}
    \end{equation*}
    is a basis of $e_{\gamma} \Z \Styl(A) e_{\gamma'}$.
\end{proof}

We remark that an alternative proof of Proposition~\ref{cartan-invariants} can
be obtained by appealing to \cite[Theorem~3.20]{DHST}, which gives
a formula for the Cartan invariants for any $\J$-trivial monoid $M$.
Applied to $\Styl(A)$, the formula says that the Cartan invariants are given by
\begin{equation*}
    \{x \in \Styl(A) : \mathrm{lfix}(x) = \gamma \text{~and~} \mathrm{rfix}(x) = \gamma'\},
\end{equation*}
where
\begin{itemize}
    \item
        $\mathrm{lfix}(x) = \min_{\leq_\J}\big\{e \in \Styl(A) : e^2 = e \text{~and~} e x = x\big\}$

    \item
        $\mathrm{rfix}(x) = \min_{\leq_\J}\big\{e \in \Styl(A) : e^2 = e \text{~and~} x e = x\big\}$
\end{itemize}
Proposition~\ref{cartan-invariants} then follows by observing that $\mathrm{lfix}(x)
= \eta(x)$ and $\mathrm{rfix}(x) = \theta(\eta(\theta(x)))$ for all $x \in
\Styl(A)$.

Finally, using a similar argument to the proof of
Proposition~\ref{cartan-invariants},
one obtains bases for the
right and left indecomposable projective $\Styl(A)$-modules.

\begin{proposition}
    For $\gamma \in \Gamma(A)$,
    \begin{enumerate}
        \item
            $\{ e_{\eta(x)} w_x : x \in \Styl(A) \text{~with~} \eta(x) = \gamma\}$ is basis
            of $e_{\gamma} \K \Styl(A)$, and

        \item
            $\{ e_{\eta(x)} w_x : x \in \Styl(A) \text{~with~} \theta(\eta(\theta(x))) = \gamma\}$ is basis
            of $\K \Styl(A) e_{\gamma}$.
    \end{enumerate}
\end{proposition}

\bibliographystyle{amsalpha}
\bibliography{references}

\end{document}